\documentclass[12pt]{amsart}
\usepackage{mathrsfs}
\usepackage{eucal}
\usepackage{color}
\usepackage[all]{xy}
\usepackage{amssymb,amsmath}

\usepackage{setspace}

\date{Nov.~5, 2013}

\calclayout
\newtheorem{dummy}{anything}[section]

\newtheorem*{thma}{Theorem A}

\newtheorem*{corc}{Corollary C}
\newtheorem*{corb}{Corollary B}
\newtheorem{lemma}[dummy]{Lemma}
\newtheorem{proposition}[dummy]{Proposition}

\newtheorem{condition}{Condition}
\theoremstyle{definition}
\newtheorem{definition}[dummy]{Definition}
  \newtheorem{example}[dummy]{Example}
  
  \newtheorem{remark}[dummy]{Remark}

  \newtheorem*{acknowledgement}{Acknowledgement}

\newcommand
{\eqncount}{\setcounter{equation}{\value{dummy}}%
\addtocounter{dummy}{1}}

\newcommand{\bZ}{\mathbb Z}

\newcommand{\CP}{\mathbb{CP}}

\newcommand{\bbZ}{\mathbb Z}
\newcommand{\HH}{V}

\DeclareMathOperator{\Homeo}{Homeo}

\DeclareMathOperator{\Hom}{Hom}

\DeclareMathOperator{\sign}{sign}
\newcommand{\la}{\langle}
\newcommand{\ra}{\rangle}
\newcommand{\BSpin}{BSpin}
\DeclareMathOperator{\Spin}{Spin}
\newcommand{\bd}{\partial}

\newcommand{\cy}[1]{\bZ/{#1}}

\begin{document}

\title{Recognizing products of surfaces and simply connected $4$-manifolds}
\author{Ian Hambleton}

\address{Department of Mathematics, McMaster University,  Hamilton, Ontario L8S 4K1, Canada}

\email{hambleton@mcmaster.ca }

\author{Matthias Kreck}
\address{Mathematisches Institut, Universit\"at Bonn, 
D-53115 Bonn, Germany}

\email{kreck@math.uni-bonn.de}

\thanks{Research partially supported by NSERC Discovery Grant A4000.
The authors would like  to thank the Max Planck Institut f\"ur Mathematik in Bonn for its hospitality and support while this work was done.}

\begin{abstract}  We give necessary and sufficient conditions for a closed smooth  $6$-manifold $N$ to be diffeomorphic to a product of a surface $F$ and a simply connected $4$-manifold $M$ in terms of basic invariants like the fundamental group and cohomological data. Any isometry of the intersection form of $M$ is realized by a self-diffeomorphism of $M \times F$.
\end{abstract}

\maketitle

\section{Introduction}
\label{sect:  introduction}
Simply-connected closed $6$-manifolds were classified by Wall \cite {wall-6mflds}, Jupp \cite{jupp1}, and \v{Z}ubr \cite{zubr1}. However, if the fundamental group is non-trivial, such complete information is not within reach of current techniques except in special cases. 

In this paper we consider the following problem:  given a closed, oriented $6$-manifold $N$, can we identify a closed, oriented surface $F$ and a simply-connected, closed $4$-manifold $M$ such that $N$ is diffeomorphic to $M \times F$~? Since simply connected $6$-manifolds are already classified, we assume from now on that $F \neq S^2$ has genus $\geq 1$, but the results remain true in the simply connected case. 
First we discuss some of the necessary conditions. 

\begin{condition}\label{one} The fundamental group $\pi_1(N)$ is isomorphic to the fundamental group of a closed, oriented surface $F$. 
\end{condition}
We choose a base-point preserving  classifying map $u\colon N \to F$ for the universal covering. Up to homotopy and choice of base points this is equivalent to choosing an isomorphism $\alpha\colon  \pi_1(N) \to \pi_1(F)$, where $u_\# = \alpha$. 

The next condition concerns the second homology group of the universal covering $H_2(\widetilde N)$, which for the product of $F$ with a simply connected $4$-manifold $M$ is a trivial module over $\pi_1(N)$ and so we require this: 

\begin{condition}\label{two} $H_2(\widetilde N)$ is  a trivial $\pi_1(N)$-module.
\end{condition}

Under this condition,  the Serre spectral sequence for the fibration over $F= K(\pi_1(N),1)$ with fibre $\widetilde  N$, implies 
 that we have an exact sequence 
$$
0 \to H^2(F) \xrightarrow{u^*} H^2(N) \xrightarrow{p^*} H^2(\widetilde N) \to 0,
$$
where $p$ is the universal covering projection. It follows that $H_2(\widetilde N)$ is a finitely-generated free abelian group.

The key to our recognition result is the observation that  the cohomology algebra of $N$ provides a  candidate for the intersection form of a closed, simply-connected topological $4$-manifold $M$. 
To identify this candidate, suppose that  Condition \ref{one} holds. Then the trilinear cup product form on $H^2(N)$ induces a well-defined symmetric bilinear form
 $$
I(N):  H^2(N)/{u^*H^2(F)} \times H^2(N)/{u^*H^2(F)} \to \mathbb Z
$$
by mapping $x$ and $y$ to $\la u^*([F]) \cup x \cup y , [N]\ra$, where $[F]\in H^2(F)$ is the cohomology fundamental class. If Condition \ref{two} holds, then 
$V:= H^2(N)/u^*H^2(F)\cong H^2(\widetilde N)$ is a finitely-generated free abelian group.

Under the assumption that $N \approx M \times F$, this form $I(N)$ is 
\emph{unimodular} and the finitely-generated free abelian group
$V$  is isomorphic to $H^2(M)$. Moreover, the form $I(N)$ and the intersection form 
$$s_M\colon H^2(M)\times H^2(M) \to \bbZ$$
are isometric.
We recall that vanishing of the Kirby-Siebenmann invariant of $M$, denoted $KS(M)$ is  a necessary and sufficient condition for $M \times F$ to be smoothable 
 (see \cite[Theorem 5.14, p.~318]{kirby-siebenmann1}).  If $M$ is a spin manifold, then this condition is assured by requiring $\sign(M) \equiv 0 \pmod{16}$. 
Thus we have:

\begin{condition}\label{three-new} The symmetric bilinear form $I(N)$ is unimodular, and $\sign I(N) \equiv 0 \pmod{16}$ if $N$ is a spin manifold. 
\end{condition}

If $I(N)$ is unimodular there exists a closed, simply-connected topological $4$-manifold $M$ with this intersection form, by the foundational results of Freedman \cite[Theorem 1.5]{freedman2}.  If $M$ is non-spin, then Freedman shows that we may assume $KS(M) = 0$. In either case, if $KS(M) = 0$ the manifold $M$ is \emph{uniquely} determined (up to homeomorphism) by its intersection form $s_M$. Moreover, the smooth structures on $M\times F$ are determined by lifts of its stable topological tangent bundle $\tau_{M\times F}$ (see \cite[Theorem 10.1, p.~194]{kirby-siebenmann1} for the precise statement).

\begin{definition} The \emph{standard smooth structure} on $M \times F$ is the one determined by product of the unique lift of $\tau_M \colon M \to BTOP$ to $BO$, together with  $\tau_F\colon F \to BO$. The lift of $\tau_M$ is unique because $TOP/O \simeq TOP/PL = K(\bbZ/2,3)$ in this range of dimensions.
\end{definition}

We then fix the standard smooth structure on $M\times F$ and take the product orientation with respect to given orientations on $M$ and $F$. This is our candidate for recognizing $N$ as the product $M \times F$.

Finally, we need some more information about the \emph{oriented} integral cohomology ring of $N$ and the Pontrjagin class $p_1(N) \in H^4(N)$. 
Let $q_1\colon M\times F \to M$ and $q_2\colon M \times F \to F$ denote the first and second factor projection maps. Note that the integral cohomology of $M \times F$ is $\bbZ$-torsion free, so any map $H^*(M \times F) \to H^*(N)$ of integral cohomology rings reduced mod 2  induces a map on $\cy2$-cohomology.

\begin{condition}\label{three}  Let $M$ be a closed, oriented, simply-connected topological $4$-manifold with $s_M \cong I(N)$ and $KS(M) = 0$.
There exists an isomorphism
$$\phi\colon H^*(M \times F) \to H^*(N)$$
of oriented integral cohomology rings. We assume that 
\begin{enumerate}
\addtolength{\itemsep}{0.2\baselineskip}
\item $\phi([M]\times [F]) = [N] \in H^6(N)$, 
\item $\phi\circ q_2^* = u^*\colon H^*(F) \to H^*(N)$, and 
\item $\phi$ preserves the second Stiefel-Whitney class:
$$ \phi(w_2(M\times F)) = w_2(N) \in H^2(N;\cy 2).$$
\item Moreover, the relation
$$\la \phi(x) \cup p_1(N), [N]\ra = \begin{cases}  3 \sign(M) & \text{\ if $x = q_2^*([F]) \in H^2(M\times F)$},\cr 0 & \text{\ if $x = q_1^*(y)$}\cr\end{cases}$$
holds for all $y \in H^2(M)$. 
\end{enumerate}
\end{condition} 
\begin{example} Unless $M = S^4$, the cohomology cohomology ring determines the Steenrod operations, and so $\phi$ preserves the 
second Stiefel-Whitney class. On the other hand, consider an oriented $4$-sphere bundle $N$ over $F$ with $w_2(N) \neq 0$. Then $N$ has the same cohomology ring as $S^4 \times F$ but is not diffeomorphic to the product. 
\end{example}

Now we are ready to formulate our main result.

\begin{thma} Let $N$ be a closed, oriented smooth $6$-manifold, and $\alpha\colon \pi_1(N) \cong \pi_1(F)$ for some closed, oriented surface $F$, such that Conditions  \textup{\ref{one}}-\textup{\ref{three-new}} hold. Suppose that  
\begin{enumerate}
\item  $M$ is the closed, simply-connected topological $4$-manifold $M$,  such that $s_M\cong I(N)$, with
  $KS(M) = 0$, and
\item  $\phi\colon H^*(M \times F) \xrightarrow{\approx} H^*(N)$ is 
 a ring isomorphism satisfying Condition \textup{\ref{three}}. 
\end{enumerate}
Then, there is an orientation and base-point preserving diffeomorphism $f\colon N \to M \times F$ such that $f_{\#} = \alpha$ and $f^* = \phi$.
\end{thma}

We can also ask which automorphisms of the second cohomology of $M \times F$ are induced by self-diffeomorphisms. In particular, we consider automorphisms of $H^2(M)$, and extend them by the identity on $H^2(F)$ via the identification:
$$(q_1^* , q_2^* )\colon H^2(M) \oplus H^2(F) \cong  H^2(M \times F).$$
 From the ring structure in cohomology, a necessary condition is that the automorphism on $H^2(M)$ is an isometry of the intersection form.

\begin{corb} Let $M$ be a closed topological $4$-manifold with $KS(M)=0$ and $F$ a closed oriented surface. Then each isometry of the intersection form of $M$  is induced by a self-diffeomorphism of $ M \times F$.
\end{corb}

\begin{proof}  There is  an automorphism $\phi$ of $H^*(M \times F)$, which on $H^2(M\times F)$ is the given isometry on $H^2(M)$ extended by the identity on $H^2(F)$. By Theorem A there is a self-diffeomorphism of $M \times F$ inducing $\phi$, and therefore the given isometry on $H^2(M)$. 
\end{proof}

\begin{remark} 
 In the case where $M$ is itself smooth, Donaldson theory (see \cite[Theorem 6]{friedman-morgan1}) provides  examples of isometries of $H^2(M)$ which cannot be realized by  self-diffeomorphisms of $M$. 
We also remark that an alternate argument can be given for Corollary B by using further results of Freedman and Kirby-Siebenmann. By \cite[p.~371]{freedman2} there is a homeomorphism $h\colon M\to M$ realizing the given isometry. Consider the $s$-cobordism 
$$W^5 := (M\times I) \cup_h (M\times I)$$
 obtained by gluing two cylinders $M \times I$ via $h$. 
Since $H^4(W, \bd W;\bbZ/2)= H^3(M;\bbZ/2) =0$, we can pick a lift of $\tau_W\colon W \to BO$ extending the standard lift of $\tau_M$ on both boundary components. 
 Taking the product of the $s$-cobordism with $F$,  we obtain an $s$-cobordism $W \times F$ with the product lift of $\tau_{W\times F}$ over $BO$. 
By Kirby-Siebenmann \cite[Theorem 10.1, p.~194]{kirby-siebenmann1}  there is a smooth structure on $W\times F$ which restricts to the standard smooth structure on both ends. 
 The $s$-cobordism theorem then gives a self-diffeomorphism of the standard smooth structure on
 $M\times F$,  realizing the given isometry. 
\end{remark}

Finally, we note that the smooth structure on $M\times F$ is actually \emph{unique} up to diffeomorphism. 

\begin{corc} Let $M$ be a closed, simply-connected topological $4$-manifold with $KS(M)=0$, and let $F$ be a closed oriented surface. Then $M\times F$ has a unique smooth structure.
\end{corc}

\begin{proof} We can apply Theorem A to the topological manifold $M \times F$ equipped with two different smooth structures. By Novikov \cite[Theorem 1]{novikov1}, we have Condition \ref{three} with $\phi = id$.
\end{proof}

\begin{remark}
The results of Kirby and Siebenmann \cite[Theorem 5.4, p.~318]{kirby-siebenmann1} show that the set of distinct smoothings of $M\times F$ is in bijection with 
$$[M\times F, TOP/O] = [M\times F, TOP/PL] =  H^3(M\times F; \cy 2),$$
since in this dimension every $PL$ manifold admits a unique smooth structure. 
 Theorem A shows that $\Homeo(M\times F)$ acts transitively on the set of smoothings.  It would be interesting to construct a corresponding homeomorphism for each $\alpha \in H^3(M \times F;\mathbb Z/2)$.
 
 Here a \emph{smoothing} is a pair $(N,h)$, where $N^6$ is a smooth $6$-manifold and $h\colon N \to M\times F$ is a homeomorphism; two smoothings $(N, h)$ and $(N', h')$ are equivalent if there exists a diffeomorphism $\varphi \colon N \to N'$ such that $h$ and $h'\circ \varphi$ are topologically isotopic.
\end{remark}

\begin{remark} The effectiveness of our recognition result in practice will depend on the difficulty of verifying Conditions \ref{three-new} and \ref{three}, but most of this is linear algebra.  After obtaining Conditions 1 and 2, one might proceed by showing that $H^*(\widetilde N)$ is isomorphic to a $4$-dimensional  algebra $\Lambda^*$, with $\Lambda^0 = \Lambda^4 = \bbZ$, $\Lambda^1 = \Lambda^3 = 0$, carrying the symmetric bilinear form $I(N)\colon  \Lambda^2\otimes \Lambda^2 \to \bbZ$ on a free abelian group $\Lambda^2 \cong \bbZ^r$. This gives the Euler characteristic equation $\chi(N) = \chi(F)\cdot (r+2)$, and shows
 that $H^3(N)\cong H^1(F) \otimes H^2(\widetilde N)$ is torsion-free.  Now Poincar\'e duality for $H^3(N)$ shows that $I(N)$ is unimodular.
  After that, it will be necessary to check that $H^*(N) \cong \Lambda^*\otimes H^*(F)$ as graded algebras, and proceed to construct a cohomology ring isomorphism $\phi\colon H^*(M\times F) \to H^*(N)$ with the required conditions on $w_2(N)$  and $p_1(N)$.

However complicated the process, at least the conditions depend only on the primary algebraic topology of $N$ and do not involve determining the full homotopy type of $N$. For example, we do not assume anything about $\pi_3(N)$. 
\end{remark}

\begin{acknowledgement} The authors would like to thank Larry Taylor and the referee for helpful comments and suggestions.
\end{acknowledgement}
\section{The normal $2$-type and normal $2$-smoothings}\label{sec:two}

For the proof we use the methods from \cite{kreck3} and assume that the reader is familiar with the basic concepts and theorems although we repeat the relevant definitions briefly.

We abbreviate $\HH:= H^2(N)/{u^*H^2(F)}\cong H^2(\widetilde N) $, and let 
 $H = \pi_2(N) = H_2(\widetilde N)$.  We have $H^2(\widetilde N) \cong \Hom_{\bbZ}(H_2(\widetilde N), \bbZ)$, so that $H  \cong \Hom_{\bbZ}(V, \bbZ) = V^*$.
The following remark is immediate  from the definitions. 
\begin{lemma}\label{identity} If
$N$ satisfies Condition \textup{\ref{three}} with respect to $M\times F$, then  $H^2(M) \cong \HH$.
\end{lemma}

 We start by determining the normal $2$-type of $N$.
 By definition, this is a fibration $B$ over $BSO$ where the homotopy groups of the fibre vanish in degree $\ge 3$ and such there is a lift of the normal Gauss map of $N$ over $B$, which is a $3$-equivalence. We have to distinguish two cases, where the symmetric bilinear form $I(N)\colon V \times V \to \bbZ$  is a even or odd. In the first case, the normal $2$-type is 
$$
p_{even}\colon B = K(H,2) \times F  \times \BSpin \to BSO,
$$ 
where the map is the composition of the projection to $\BSpin$ and the canonical projection to $BSO$. If the form $I(N)$ is odd, one chooses a primitive characteristic element $v \in V$,  and a complex line bundle $L_v$ over $K(H,2)$ with first Chern class $v$.  Then the normal $2$-type is 
$$
p_{odd}\colon B = K(H,2) \times F  \times \BSpin \to BSO,
$$
where $p_{odd}$ is the map given by the projection to $K(H,2) \times \BSpin$ composed by the map given by the Whitney sum of line bundle $L_v$ and the canonical map to $BSO$ (of course, we have to replace this map by a fibration). 

\begin{lemma}\label{lem:2smoothings}
The normal $2$-types of $M\times F$ and $N$ are given by $(B, p_{even})$, if $M$ is spin, or $(B,p_{odd})$ if $M$ is non-spin.
\end{lemma}
\begin{proof}
 We first look at the second stage of the Postnikov tower of $N$, this is  a fibration over $K(\pi_1(N),1)$ with fibre $K(\pi_2(N),2)$, where in our situation $\pi_2(N) = H$. These fibrations are classified by the action of $\pi_1(N)$ on $\pi_2(N)$ and the $k$-invariant $k \in H^3(\pi_1(N); \pi_2(N) )$. This group is zero, and so the action of $\pi_1(N)$ on $\pi_2(N)$ determines the Postnikov tower. If the $\pi_1(N)$-action is trivial, then we have the trivial fibration. 
Next, we use our data to construct  a $3$-equivalence
$$c_{M \times F}: = g_{M \times F} \times h_{M \times F}\colon M \times F \to K(H,2) \times F,$$
and a $3$-equivalence
$$c_N:= g_{N} \times h_{N}\colon N \to K(H,2) \times F,$$
 which is compatible with our data $\alpha$ and $\phi$. For this we consider the map $$g_{M \times F}\colon M \times F \to K(H,2)$$ such that  
 $(g_{M \times F})^*\colon  V  \to H^2(M\times F) = H^2(M) \oplus H^2(F) = V \oplus H^2(F)$ is the inclusion onto the first summand (see Lemma \ref{identity}), and   choose a base point preserving map $g_N\colon  N \to K(H,2)$ such that $(g_N)^* = \phi \circ (g_{M\times F})^*$. Then we consider the projection $h_{M \times F}=q_2 \colon  M \times F \to F$ and $h_N = u\colon  N \to F$. From Conditions \ref{one} - \ref{three} it is clear that
 the maps $c_{M \times F}$ and $c_N$ are $3$-equivalences, with $(c_N)^*= \phi\circ (c_{M\times F})^*$. 
 
If $N$ is  Spin-manifold, then by assumption $M \times F$ is a Spin-manifold and we equip both manifolds with an arbitrary Spin structure $\omega_N$ and $\omega_{M \times F}$. If $N$ and so $M \times F$ are not  Spin-manifolds, then we choose a  primitive  class $v \in H^2(M \times F;\mathbb Z)$, such that its component in $H^2(F;\mathbb Z)$ is zero, which reduces to $w_2(M \times F)$ and a spin structure $\omega_{M \times F}$ on $\nu (M \times F) \oplus L_v$, where $L_v$ is the complex line bundle classified by $v$. Similarly, we choose a Spin structure $\omega_{N}$ on $\nu(M) \oplus L_{\phi(v)}$. The maps $c_{M \times F}$ and $c_N$ together with the (twisted) Spin-structures are normal 2-smoothings in $(B, p_{odd/even})$. 
\end{proof}

\section{The bordism groups}

 The next step in the proof of Theorem A is to show that, under the given conditions, the  normal $2$-smoothings constructed in Section \textup{\ref{sec:two}} are bordant in $\Omega_6(B, \xi)$, where $\xi$ is the bundle classified by $p_{odd}$ or $p_{even}$ depending on the normal 2-type.
 
 The method of proof is based on detecting elements in the bordism group by explicit invariants. We have $H \cong \bbZ^r$ so that $K := K(H,2) = (\CP^\infty)^r$. Let $Dp_1(N) \in H_2(N)$ denote the Poincar\'e dual of the first Pontrjagin class.

 \begin{proposition} There is an injection
 $\Omega_6(B,\xi) \to \bbZ \oplus H_6(K) \oplus H_4(K)\oplus H_2(K)$,
 given 
 by $\sign I(N)$, and the images of $[N]$, $[N]\cap u^*([F])$, $Dp_1(N)$ under the reference maps $c_N\colon N \to B$ for the normal $2$-types.
 \end{proposition}
To compute the bordism groups we consider the functor associating to a space $X$ the bordism group of $p_{odd/even}\colon  X \times   K(H,2) \times \BSpin \to BSO$, where the maps are defined as above in the case $X = F$. This is a homology theory  denoted by $h_k(X)$ and so we can use the Mayer-Vietoris sequence to compute it,  by writing a surface of  genus $g$ as $D_2 \cup Y$, where $Y$ is a wedge of $2g$ circles. Then we obtain an exact sequence
$$
 \tilde h_7(S^2) \to  \tilde h_6(Y) \to \tilde h_6(F) \to \tilde h_6(S^2) \to \tilde h_5(Y),
$$
or, if we apply the suspension isomorphism, the exact sequence: 
\eqncount
\begin{equation}\label{eqn:threeone}
h_5(pt) \to \sum _{2g} h_5(pt)  \to \tilde h_6(F) \to h_4(pt) \to .
\end{equation}
The map from $h_6(F)$ to $h_4(pt)$ is defined by sending $[N, c_N] \mapsto [Q, c_Q]$, 
where $c_N\colon  N \to B$ is a lift of the normal Gauss map, and $Q \subset N$ is the pre-image of a regular value of the composition of the map to $B$ with the projection to $F$. The reference map $c_Q \colon Q \to B$ is given by the restriction of $c_N$ to $K:=K(H,2)$, together with the induced bundle  and  (twisted) Spin-structure. 

\medskip
To proceed further we need information about 
$h_k(pt) = \Omega^{\Spin}_k((\CP^\infty)^r)$, for $p_{even}$, and 
$h_k(pt) = \Omega^{\Spin}_k((\CP^\infty)^r, L)$, for $p_{odd}$.
We begin with the case $r=1$. 
\begin{lemma}\label{lem:rone} Let $L$ denote the Hopf bundle over $\CP^\infty$.
\begin{enumerate}
\item The map 
$\Omega_4^{\Spin}(\CP^\infty )\to \mathbb Z \oplus \mathbb Z$ given by 
the signature and the image of the fundamental class is injective. 
\item $\Omega_6^{\Spin}(\CP^\infty) = \mathbb Z \oplus \mathbb Z$,
detected by the image of the fundamental class  and the image $Dp_1(N)$, 
\item $ \Omega_4^{\Spin}(\CP^\infty;L)\cong \bbZ \oplus \bbZ$, detected by the image of the fundamental class and the image of 
$Dp_1(N)$.
\end{enumerate}
\end{lemma}
\begin {proof} 
The $E^2$-term of the Atiyah-Hirzebruch spectral sequence  computing $\Omega_4^{\Spin}(\CP^\infty)$ gives $\mathbb Z$ in position $(0,4)$ and $(4,0)$, and $\mathbb Z/2$ in position $(2,2)$.
 The differential 
 $$d\colon  H_4(\CP^\infty;\mathbb Z) \to H_2(\CP^\infty;\mathbb Z/2)$$
  is the reduction mod $2$ composed by the dual of $Sq^2$ \cite[Proposition 1, p.~750]{teichner1} and so is nontrivial. This implies that 
$$
\Omega_4^{\Spin}(\CP^\infty )\to \mathbb Z \oplus \mathbb Z
$$
given by the signature and the image of the fundamental class is injective. 

Analyzing the Atiyah-Hirzebruch spectral sequence for $\Omega_6^{\Spin}(\CP^\infty)$ gives an entry $\mathbb Z$ at position $(2,4)$ and $(6,0)$ and $\mathbb Z/2$ at position $(4,2)$. This time the differential vanishes and so the bordism group is either $\mathbb Z \oplus \mathbb Z$ or $\mathbb Z \oplus \mathbb Z \oplus \mathbb Z/2$. It was proven in \cite[p.~258]{kreck5} that
$$\Omega_6^{\Spin}(\CP^\infty) = \mathbb Z \oplus \mathbb Z,$$
detected by the image of the fundamental class  and the image of $Dp_1(N)$.

Now we consider the bordism groups twisted by the line bundle $L$.  We reduce the $4$-th bordism group  to the untwisted case by using the  isomorphism given by taking the transversal preimage of  $\CP^{N-1}$, where we replace $\CP^\infty$ by $\CP^N$ for a large $N$:
$$
\Omega_6^{\Spin}(\CP^\infty )\cong \Omega_4^{\Spin}(\CP^\infty;L)
$$
(here we use that $\Omega_6^{\Spin} = \Omega_5^{\Spin} = 0$)
implying that $\Omega_4^{\Spin}(\CP^\infty;L)  \cong \mathbb Z \oplus \mathbb Z$ again detected by the image of the fundamental class and the signature. 

Finally the computation of $\Omega_4^{\Spin}(\CP^\infty;L) \cong \mathbb Z \oplus \mathbb Z$, again detected by the image of the fundamental class and the image of $Dp_1(N)$, follows from the Atiyah-Hirzebruch spectral sequence. This time the $E^\infty $-term is torsion free in the $6$-line, since the differential $d\colon  H_6(\CP^\infty;\mathbb Z) \to H_4(\CP^\infty;\mathbb Z/2)$ is the reduction mod $2$ composed by the dual of $Sq^2$ plus $c_1(L) \cup ...)$ (see again \cite[Proposition 1, p.~750]{teichner1}) and so is trivial.
\end{proof}

\begin{lemma}\label{lem:threeone} $h_5(pt)$ is zero.
  The map $h_6(pt) \to H_6(K) \oplus H_2(K)$ given by the image of the fundamental class and the  image of $Dp_1(N)$  is injective.
 The map given by the signature and the image of the fundamental class is an injection $h_4(pt) \to \mathbb Z \oplus H_4(K)$.
\end{lemma}

\begin{proof}
Now we consider the general case.  If we show that the bordism groups are again torsion free, then the statements follow  from the Atiyah-Hirzebruch spectral sequence. We first note that by applying an appropriate isomorphism of $H\cong \bbZ^r$ we can assume in the twisted case that $c_1(L) = (0,...,0,1)$. With this we write $(\CP^\infty)^r = X \times \mathbb \CP^\infty$ and compute $\Omega_k^{\Spin}(X \times \CP^\infty)$ and $\Omega_k^{\Spin} (X \times \CP^\infty;L)$ for $k=4$ and $6$, where $X = (\CP^\infty)^{r-1}$ and $L$ is the Hopf bundle over the last factor. We assume inductively that $\Omega_k(X)$ is torsion free for $k = 4$ and $k=6$. Using again the transversal preimage of  $\CP^{N-1}$, where we replace $\CP^\infty$ by $\CP^N$ for a large $N$,  we have an exact Gysin sequence (see \cite[Section I.6, p.~315]{dax1}, \cite{salmonsen1}): 
$$
\Omega_5^{\Spin}(X \times \CP^\infty;L) \to \Omega_6^{\Spin}(X) \to \Omega_6^{\Spin}(X \times \CP^\infty) \to \Omega_4^{\Spin}(X \times \CP^\infty;L).
$$
Since the odd dimensional groups are by the Atiyah-Hirzebruch spectral sequence torsion, we see that $\Omega_6^{\Spin}(X \times \CP^\infty)$ is torsion free, if 
$\Omega_4^{\Spin}(X \times \CP^\infty;L)$ is torsion free. For this we consider the corresponding exact Gysin sequence (again, see \cite[Section I.6]{dax1}):
$$
\Omega_3^{\Spin}(X \times \CP^\infty;L\oplus L) \to \Omega_4^{\Spin}(X) \to \Omega_4^{\Spin}(X \times \CP^\infty;L) \to \Omega_2^{\Spin}(X \times \CP^\infty;L\oplus L).
$$
The Atiyah-Hirzebruch spectral sequence implies that 
$$ \Omega_2^{\Spin}(X \times \CP^\infty;L\oplus L) \cong H_2(X\times \CP^\infty) \oplus \mathbb Z/2.$$
 Now we compare this exact sequence with that for $X$ a point:
$$
\Omega_3^{\Spin}( \CP^\infty;L\oplus L) \to \Omega_4^{\Spin} \to \Omega_4^{\Spin}( \CP^\infty;L) \to \Omega_2^{\Spin} (\CP^\infty;L\oplus L).
$$
We have maps from the first to the second exact sequence given by the projection from $X$ to a point. Now suppose that $\Omega_4^{\Spin}(X \times \CP^\infty;L)$ contains a torsion element. Then, since by assumption $\Omega_4^{\Spin}(X)$ is torsion free, this maps to the non-trivial torsion element in $\Omega_2^{\Spin}(X \times \CP^\infty;L\oplus L)$. But then the image in $\Omega_4^{\Spin}( \CP^\infty;L)$ is again a non-trivial torsion element, since in $\Omega_2^{\Spin} (\CP^\infty;L\oplus L)$ it maps to the non-trivial element. But this is a contradiction to what we have shown above that $\Omega_4^{\Spin}( \CP^\infty;L)$ is torsion free. 

Now we have shown half of our statements, namely that $\Omega_6^{\Spin}(X \times \CP^\infty)$ is torsion free as well as
$\Omega_4^{\Spin}(X \times \CP^\infty;L)$. We prove the other cases by a similar argument using this time the exact Gysin sequences:
$$
\Omega_5^{\Spin}(X \times \CP^\infty;L^{\oplus 2}) \to \Omega_6^{\Spin}(X) \to \Omega_6^{\Spin}(X \times \CP^\infty;L) \to \Omega_4^{\Spin}(X \times \CP^\infty;L^{\oplus 2})
$$
and 
$$
\Omega_3^{\Spin}(X \times \CP^\infty;L^{\oplus 3}) \to \Omega_4^{\Spin}(X) \to \Omega_4^{\Spin}(X \times \CP^\infty;L^{\oplus 2}) \to \Omega_2^{\Spin}(X \times \CP^\infty;L^{\oplus 3}). $$

This case is easier since $\Omega_2^{\Spin}(X \times \CP^\infty;L^{\oplus 3})$ is torsion free, the torsion in the $E^2$ term is killed by the $d_2$-differential. 

Finally we show that $\Omega_4^{\Spin}(X \times \CP^\infty)$ is torsion free using the exact sequence:
$$
\Omega_3^{\Spin}(X \times \CP^\infty;L) \to \Omega_4^{\Spin}(X) \to \Omega_4^{\Spin}(X \times \CP^\infty) \to \Omega_2^{\Spin}(X \times \CP^\infty;L).
$$
By the same argument as above the group $\Omega_2^{\Spin}(X \times \CP^\infty;L)$ is torsion free finishing the argument. 

Now we show that $h_5(pt) = 0$. On the line corresponding to $h_5(pt)$ the only non-trivial entry in the $E_2$-term is $H_4(K;\mathbb Z/2)$. If $I(N)$ is even, the differentials are even given by the dual of $Sq^2$. If  $I(N)$ is odd, where we had to use twisted Spin-structures, the differentials are given by the dual of $Sq^2$ plus $ x \mapsto Sq^2x + w_2 \cup x$, where $w_2$ is the reduction of $c$ mod $2$. It is an easy exercise to show that the $E^3$-term is zero in both cases. 
\end{proof}

\medskip
With this information we show that the bordism classes of $N$ and $M \times F$, equipped with the normal $2$-smoothings  constructed in Section \ref{sec:two}, agree when identified via the maps  $\alpha$ and $\phi$.  By the exact sequence (\ref{eqn:threeone}) and Lemma \ref{lem:threeone},  this amounts to showing (i) the bordism classes in $h_6(pt)$ agree,  and (ii)  that the classes in $h_4(pt)$ agree, which we obtain as transversal preimages of a regular value of the map to $F$ given by composing our normal $2$-smoothings with the projection to $F$. 

By Lemma \ref{lem:threeone}, the first invariant is  given by two invariants, the image of the fundamental class in $H_6(K(H,2))$ and the image of $Dp_1(N)$ in $H_2(K(H,2))$. The image of the fundamental class in $H_6(K(H,2))$ is (by the cohomological structure of $K(H,2)$) equivalent to the triple product $g^*(x) \cup g^*(y) \cup g^*(z)$ for classes $x,y,z$ in $H^2(K(H,2))$.   But these products vanish for $M \times F$ with $g= g_{M\times F}$, and for $N$ with $g= g_N$, since $\phi$ is an isometry of the cohomology rings and $(g_N)^* = \phi \circ (g_{M\times F})^*$. The image of $Dp_1(N)$ in $H_2(K(H,2))$ is determined by the products $g^*(x) \cup p_1$ for all $x \in H^2(K(H,2))$ and vanishes for $M \times F$ and for $N$  by Condition \ref{three}. 

Thus we are left with the invariant in $h_4(pt)$. Let $Q \subset N$ be the transversal preimage of a regular value of the map $u\colon  N \to F $. By Lemma \ref{lem:threeone}, bordism classes in $h_4(pt)$ are determined by the signature of the underlying $4$-manifold, and  the image of the fundamental class $[Q]$ in $H_4(K(H,2)$. For a class  $\beta \in H^4(N)$ we have the adjunction formula
$$
\la u^*([F]) \cup \beta, [N]\ra = \la i^*(\beta), [Q] \ra,
$$
where $i\colon Q \to N$ is the inclusion.
Applying this to $\beta = p_1(N)$  we obtain:
$$
\la p_1(N) \cup u^*([F]) , [N]\ra = \la p_1(Q),[Q] \ra,
$$
since the normal bundle of $Q$ is trivial.
The signature theorem for $Q$ and Condition \ref{three} (iv) imply that 
$$
\la p_1(N) \cup u^*([F]) , [N]\ra = 3 \sign (Q) = 3\sign(M),
$$
proving the equality for the first invariant in $h_4(pt)$. 

For the second invariant we note that the image of the fundamental class of $Q$ in  $H_4(K(H,2)$ is determined by the numbers 
$$
\la i^*g^*(x)\cup i^*g^*(y), [Q] \ra.
$$
We apply again the adjunction formula for  $\beta = g^*(x) \cup g^*(y)$, and get
$$
\la g^*(x) \cup g^*(y) \cup u^*([F])  , [N]\ra = \la i^*g^*(x)\cup i^*g^*(y), [Q] \ra,
$$
where $g= g_N$. A similar formula holds for $M \times F$ and $g=g_{M\times F}$. The left side agrees for $N$ and $M \times F$, since $\phi$ is an isomorphism of the cohomology ring. Thus also the second invariant for the element in $h_4(pt)$ agrees. 
Summarizing, we have shown: 

\begin {proposition}\label{prop:bordismclasses} If the conditions of Theorem A  are fulfilled,  then the bordism classes  
$$[N, c_N] = [M\times F, c_{M\times F}] \in \Omega_6(B, \xi),$$ for  the normal $2$-smoothings on $N$ and $M \times F$
constructed in Lemma \textup{\ref{lem:2smoothings}}. 
\end {proposition}

\section {The proof of Theorem A}

 We consider $N$ and $M \times F$ equipped with normal $2$-smoothings compatible with $\alpha$ and $\phi$. 
 By  Proposition \ref{prop:bordismclasses},  the corresponding bordism classes are equal. Choose a $B$-bordism $W$ between these two normal $2$-smoothings. Since the Euler characteristics of $N$ and $M \times F$ agree, there is an obstruction $\theta(W) \in l_7(\pi_1(F))$ which is elementary if and only if $W$ is $B$-bordant to an s-cobordism. We first note that the Whitehead group for $\pi_1(F)$ vanishes by a result of Farrell-Hsiang \cite{farrell-hsiang1a}, so that we can ignore decorations in the $l$-monoids and $L$-groups. Next we note that  the intersection form  on $\pi_3(M\times F)\cong \pi_3(M)$  with values in the group ring 
vanishes identically (since $\Hom_{\bbZ G}(\bbZ, \bbZ G) = 0$ for $G$ an infinite group). 
 By \cite[Proposition 8, p.~739]{kreck3},  this implies that $\theta (W)$ sits in the ordinary $L$-group $L_7(\pi_1(F))$. But by Cappell \cite[Theorem 18]{cappell.s.e.1973.1}, 
  there is a closed $7$-manifold with $B$-structure so that after taking the disjoint union of $W$ with this manifold the obstruction in $L_7(\pi_1(F))$ vanishes. This completes the proof.


\providecommand{\bysame}{\leavevmode\hbox to3em{\hrulefill}\thinspace}
\providecommand{\MR}{\relax\ifhmode\unskip\space\fi MR }
\providecommand{\MRhref}[2]{%
  \href{http://www.ams.org/mathscinet-getitem?mr=#1}{#2}
}
\providecommand{\href}[2]{#2}

\end{document}